\documentclass[12pt]{amsart}

\usepackage{amscd, amsfonts, amsmath, amssymb, amsthm, enumerate, flexisym, float, graphicx, hyperref, pdfpages,tikz,mathtools, comment}
\usepackage{tikz-cd}
\makeatletter
\@namedef{subjclassname@2020}{Mathematics Subject Classification \textup{2020}}

\newsavebox{\@brx}
\newcommand{\llangle}[1][]{\savebox{\@brx}{\(\m@th{#1\langle}\)}%
  \mathopen{\copy\@brx\kern-0.5\wd\@brx\usebox{\@brx}}}
\newcommand{\rrangle}[1][]{\savebox{\@brx}{\(\m@th{#1\rangle}\)}%
  \mathclose{\copy\@brx\kern-0.5\wd\@brx\usebox{\@brx}}}

\makeatother

\newtheorem{theorem}{Theorem}[section]
\newtheorem{corollary}[theorem]{Corollary}
\newtheorem{lemma}[theorem]{Lemma}
\newtheorem{proposition}[theorem]{Proposition}

\theoremstyle{definition}

\newtheorem{remark}[theorem]{Remark}

\numberwithin{equation}{subsection}

\newcommand{\Env}{\operatorname{Env}}

\newcommand{\Inn}{\operatorname{Inn}}

\newcommand{\C}{\operatorname{C}}

\newcommand{\mcg}{\mathcal{M}}
\newcommand{\dq}{\mathcal{D}}

\begin{document}

\title{Finiteness of canonical quotients of Dehn quandles of surfaces}
\author{Neeraj K. Dhanwani}
\author{Mahender Singh}

\address{Department of Mathematical Sciences, Indian Institute of Science Education and Research (IISER) Mohali, Sector 81, S. A. S. Nagar, P. O. Manauli, Punjab 140306, India.}
\email{neerajk.dhanwani@gmail.com}
\email{mahender@iisermohali.ac.in}


\begin{abstract}
The Dehn quandle of a closed orientable surface is the set of isotopy classes of non-separating simple closed curves with a natural quandle structure arising from Dehn twists. In this paper, we consider finiteness of some canonical quotients of these quandles. For a surface of positive genus,  we give a precise description of the 2-quandle of its Dehn quandle. Further, with some exceptions for genus more than two, we determine all values of $n$ for which the $n$-quandle of its Dehn quandle is finite. The result can be thought of as the Dehn quandle analogue of  a similar result of Hoste and Shanahan for link quandles. We also compute the size of the smallest non-trivial quandle quotient of the Dehn quandle of a surface. Along the way, we prove that the involutory quotient of an Artin quandle is precisely the corresponding Coxeter quandle and also determine the smallest non-trivial quotient of a braid quandle.
\end{abstract}
\maketitle

\section{Introduction}
The Dehn quandle of a group with respect to a subset is the union of conjugacy classes of elements of the set equipped with the operation of conjugation. It turns out that these are precisely the quandles that embed into their enveloping groups, and many well-known classes of quandles satisfy this property \cite{Dhanwani-Raundal-Singh-2021}. The first being the class of free quandles, where a free quandle on a set is simply the union of conjugacy classes of the free generating set in the corresponding free group. The second interesting class of examples, which is also the focus of this paper, is given by surfaces. Let $\mathcal{M}_g$ be the mapping class group of a closed orientable surface $S_g$ of genus $g \ge 1$ and $\mathcal{D}_g^{ns}$ the set of isotopy classes of non-separating simple closed curves on $S_g$. It is well-known that $\mathcal{M}_g$ is generated by Dehn twists along finitely many simple closed curves from $\mathcal{D}_g^{ns}$ \cite[Theorem 4.1]{Farb-Margalit2012}. The binary operation $\alpha * \beta= \tau_\beta(\alpha),$ where $\alpha, \beta \in \mathcal{D}_g^{ns}$ and $\tau_\beta$ is the Dehn twist along $\beta$, turns $\mathcal{D}_g^{ns}$ into a quandle called the Dehn quandle of the surface $S_g$. In fact, $\mathcal{D}_g^{ns}$ can be seen as a subquandle of the conjugation quandle of $\mathcal{M}_g$, by identifying the isotopy class of a non-separating simple closed curve with the isotopy class of its corresponding Dehn twist. These quandles originally appeared in the work of Zablow \cite{Zablow1999, Zablow2003}.  It turns out that two surfaces of genus more than two are homeomorphic if and only if their Dehn quandles are isomorphic \cite[Proposition 6.5]{Dhanwani-Raundal-Singh-2021}. A homology theory based on Dehn quandles of surfaces has been derived in  \cite{Zablow2008} and it has been shown that isomorphism classes of Lefschetz fibrations over a disk correspond to quandle homology classes in dimension two. In a related direction, the papers \cite{kamadamatsumoto,Yetter2002, Yetter2003} considered a quandle structure on the set of isotopy classes of simple closed arcs on an orientable surface with at least two punctures, referred as the quandle of cords. In the case of a disk with $n$ punctures, this quandle is simply the Dehn quandle of the braid group $B_n$ on $n$ strands with respect to its standard set of generators. It turns out that Dehn quandles of groups with respect to their subsets include many well-known constructions of quandles from groups including conjugation quandles of groups, free quandles, Coxeter quandles, Dehn quandles of closed orientable surfaces, quandles of cords of orientable surfaces, knot quandles of prime knots, core quandles of groups and generalized Alexander quandles of groups with respect to fixed-point free automorphisms (see \cite{Dhanwani-Raundal-Singh-2021} for more details). A presentation for the quandle of cords of the plane and the 2-sphere has been given in \cite{kamadamatsumoto}. Beyond these cases, not much was known about presentations of Dehn quandles until the recent work \cite{Dhanwani-Raundal-Singh-2022}, wherein two approaches to write explicit presentations for Dehn quandles using presentations of their underlying groups have been given. Several computations of presentations have been given, including Dehn quandles of spherical Artin groups, surface groups and mapping class groups of orientable surfaces for small genera.
\par

Since link quandles are always infinite except for the unknot and the Hopf link, it is natural to explore finiteness of $n$-quandles of link quandles. It had been conjectured by Przytycki that the $n$-quandle of an oriented link $L$ in the $3$-sphere is finite if and only if the fundamental group of the $n$-fold cyclic branched cover of the $3$-sphere, branched over $L$, is finite. The conjecture has been proved by Hoste and Shanahan  in \cite{MR3704243}, wherein they use Dunbar's classification \cite{MR0977042} of spherical $3$-orbifolds to determine all links with a finite $n$-quandle for some $n$.  In this paper, we consider the analogous problem for Dehn quandles of surfaces. The two families of quandles have a rather curious intersection. Thanks to the work \cite{NiebrzydowskiPrzytycki2009} of Niebrzydowski and Przytycki where they proved that the knot quandle of the trefoil is isomorphic to the Dehn quandle of the torus. The problem can be thought of as a quandle analogue of \cite[Problem 28]{MR0375281}, which asks whether the normal closure of squares of Dehn twists in the mapping class group of a closed orientable surface is of finite index. See also \cite{Humphries1992, Funar1999} for further work on the problem.
\par

The paper is organised as follows. In Section \ref{sec-n-quandles-dehn-quandles}, we establish some general results on finiteness of $n$-quandles of Dehn quandles. We prove that if the $n$-quandle of the Dehn quandle of a group is finite, then certain canonical quotient of the group is finite (Theorem \ref{finiteness of g_n}). We consider Dehn quandles of Artin groups in Section \ref{inv quotients of Artin quandle} and prove that the involutory quandle of an Artin quandle is the corresponding Coxeter quandle (Theorem \ref{Artin to Coxeter quandle}). We also determine the size of the smallest non-trivial quandle quotient of the Dehn quandle of the braid group with respect to its standard generating set  (Proposition \ref{bound of braid quandle image}). In Section \ref{inv quotients of Dehn quandle}, we discuss finiteness of $n$-quandles of Dehn quandles of surfaces. For a closed orientable surface of positive genus, we give a precise description of its 2-quandle (Theorem \ref{dg-iso-pg2}). Further, with some exceptions for genus more than two, we determine all values of $n$ for which the $n$-quandle of the Dehn quandle is finite (Theorem \ref{finiteness n-quandle surface}). As a final result, we also determine the size of the smallest non-trivial quandle quotient of the Dehn quandle of a surface (Proposition \ref{smallest quotient of dehn surface}).
\medskip

\section{Preliminaries}\label{section-prelim}
Throughout the paper, we consider right-distributive quandles. Recall that, a {\it quandle} is a non-empty set $Q$ equipped with a binary operation $*$ satisfying the following axioms:
\begin{enumerate}[(i)]
\item $x*x=x$\, for all $x\in Q$.
\item For each $x, y \in Q$, there exists a unique $z \in Q$ such that $x=z*y$.
\item $(x*y)*z=(x*z)*(y*z)$\, for all $x,y,z\in Q$.
\end{enumerate}
\par
Topologically, the three quandle axioms correspond to the three Reidemeister moves of planar diagrams of links in the 3-space. Though quandles are ubiquitous, groups and knots are prominent sources of quandles. For instance, if $G$ is a group, then the binary operation $x*y=y x y^{-1}$ turns $G$ into the quandle called the \textit{conjugation quandle} of $G$. Concerning knot theory, every link can be assigned a quandle called the \textit{link quandle} which is a complete invariant of non-split links up to weak equivalence. This fundamental result appeared independently in the works of Joyce \cite{Joyce1979, Joyce1982} and Matveev \cite{Matveev1982}, and has led to much of the recent works on the subject. 
\par

Morphisms and automorphisms of quandles are defined in the obvious way. Note that the quandle axioms are equivalent to saying that for each $y\in Q$, the map $S_y:Q\to Q$ given by $S_y(x)=x* y$ is an automorphism of $Q$ fixing $y$. The group $\Inn(Q)$ generated by such automorphisms is called the group of inner automorphisms of $Q$. The group $\Inn(Q)$ acts on the quandle $Q$ and the corresponding orbits are referred as connected components. For example, knot quandles are connected, whereas link quandles of split links are not connected.
\par

The \textit{enveloping group} $\Env(Q)$ of a quandle $Q$ is the group with the set of generators as $\{e_x~|~x \in Q\}$ and the defining relations as
\begin{equation*}
e_{x*y}=e_y e_x e_y^{-1}
\end{equation*}
for all $x,y\in Q$. For example, the enveloping group of the link quandle of a link in the 3-space is the corresponding link group \cite{Joyce1979, Joyce1982}. The natural map
\begin{equation*}
\eta: Q \to \Env(Q)
\end{equation*}
given by $\eta(x)=e_x$ is a quandle homomorphism with $\Env(Q)$ viewed as the conjugation quandle. The map $\eta$ is not injective in general. In fact, Dehn quandles are precisely the ones for which this map is injective.  The functor from the category of quandles to that of groups assigning the enveloping group to a quandle is left adjoint to the functor from the category of groups to that of quandles assigning the conjugation quandle to a group. Thus, enveloping groups play a crucial role in understanding of quandles themselves. 
\par

Using defining axioms  \cite[Lemma 4.4.7]{MR2634013}, any element of a quandle can be written in a left-associated product of the form
\begin{equation*}
\left(\left(\cdots\left(\left(a_0*^{e_1}a_1\right)*^{e_2}a_2\right)*^{e_3}\cdots\right)*^{e_{n-1}}a_{n-1}\right)*^{e_n}a_n,
\end{equation*}
which, for simplicity, we write as
\begin{equation*}
a_0*^{e_1}a_1*^{e_2}\cdots*^{e_n}a_n.
\end{equation*}

Let $n \ge 2$ be an integer. A quandle $Q$ is called an {\it $n$-quandle} if $$x*\underbrace{y *y*\cdots *y}_{n ~\mathrm{times}}=x$$ for all $x, y \in Q$. Equivalently, a quandle is an $n$-quandle if and only if each basic inner automorphism $S_x$ has order dividing $n$.  A 2-quandle is also called {\it involutory}. For example, the core quandle of any group is involutory.
\medskip

\section{Finiteness of $n$-quandles of Dehn quandles }\label{sec-n-quandles-dehn-quandles}
Let $G$ be a group, $A$ a non-empty subset of $G$ and $A^G$ the set of all conjugates of elements of $A$ in $G$. The {\it Dehn quandle} $\mathcal{D}(A^G)$ of $G$ with respect to $A$ is defined as the set $A^G$ equipped with the binary operation of conjugation, that is, $$x*y=yxy^{-1}$$ for all $x, y \in \mathcal{D}(A^G)$. The class of Dehn quandles contains many well-known constructions of quandles from groups. Notice that $\mathcal{D}(G^G)$ is the conjugation quandle of $G$. If $F(S)$ is the free group generated by $S$, then $\mathcal{D}(S^{F(S)})$ is the free quandle on $S$. If $(W, S)$ is a Coxeter system, then $\mathcal{D}(S^{W})$ is the so called Coxeter quandle \cite{MR4175808, Nosaka2017}. Let $S_{g}$ be a closed orientable surface of genus $g$ and $\mathcal{M}_{g}$ its mapping class group. If $S$ is the set of Dehn twists about non-separating simple closed curves, then $\mathcal{D}(S^{\mathcal{M}_{g}})$ is the Dehn quandle of the surface \cite{Zablow1999, Zablow2003, kamadamatsumoto}. 
\par

Given a quandle $Q$ and an integer $n \ge2$, the $n$-quandle $(Q)_n$ of $Q$ is defined as the quotient of $Q$ by the relations $$x*^n y:=x*\underbrace{y*y*\cdots*y}_{n~ \mathrm{times}}=x$$ for all $x,y \in Q$. Notice that there is a natural epimorphism $\pi: Q \to (Q)_n$. Further, the construction satisfies the universal property that for any quandle homomorphism $\phi:Q \to Y$, where $Y$ is a $n$-quandle, there exists a unique quandle homomorphism $\bar{f}: (Q)_n \to Y$ such that $\bar{f} ~\pi= f$. The construction is of particular interest from the point of view of knot theory. Since link quandles are always infinite except for the unknot and the Hopf link, it is interesting to explore finiteness of $n$-quandles of link quandles. Using Thurston's geometrization theorem and Dunbar's classification of spherical 3-orbifolds, a complete classification of links whose $n$-quandles are finite for some $n$ has been given in \cite{MR3704243}.
\par

 We begin with the following general observation.
 \par
 
\begin{proposition}\label{lem:presen_n_quandle}
Let $Q$ be a quandle with a presentation $\langle S \mid R \rangle$. Then, for each $n \ge 2$, $(Q)_n$ has a presentation $\langle S \mid R \cup T \rangle$, where $T=\{x*^n y=x \mid x,y\in S\}$.
\end{proposition}

\begin{proof}
Let $n \ge 2$ and $X$ be the quandle with presentation $\langle S \mid R\, \cup\, T \rangle$. First, we claim that $X$ is a $n$-quandle. Consider elements $x\in X$ and $y\in S$. We write $x=a_1*^{\epsilon_1}a_2*^{\epsilon_2}\cdots*^{\epsilon_{l-1}}a_l$, where $a_i\in S$ and $\epsilon_i\in \{1,-1\}$. By right distributivity, we obtain $x*y=(a_1*y)*^{\epsilon_1}(a_2*y)*^{\epsilon_2}\cdots*^{\epsilon_{l-1}}{(a_l*y)}$ and a repeated application of the property further gives $$x*^ny=(a_1*^ny)*^{\epsilon_1}(a_2*^ny)*^{\epsilon_2}\cdots*^{\epsilon_{l-1}}{(a_l*^ny)}=a_1*^{\epsilon_1}a_2*^{\epsilon_2}\cdots*^{\epsilon_{l-1}}a_l=x,$$
which is desired.
\par
Now take $x, y\in X$, and use induction on the word length of $y$. When  $y$ is of word length one,  we are in the preceding case. We assume that $y=a_1*^{\epsilon_1}a_2*^{\epsilon_2}\cdots*^{\epsilon_{k-1}}a_k$, where $k>1$. Then we write $y=y_1*^{\epsilon_{k-1}}a_k$, where $y_1:=a_1*^{\epsilon_1}a_2*^{\epsilon_2}\cdots*^{\epsilon_{k-2}}a_{k-1}$ is of word length less than $k$. By induction hypothesis, the assertion holds for all words of length less than $k$. Then, using \cite[Lemma 4.4.7]{MR2634013}, we have
\begin{align*}
x*^ny &=x*^n\left(y_1*^{\epsilon_{k-1}}a_k \right) \\
&=\left(x*^{-\epsilon_{k-1}} a_k *y_1 *^{\epsilon_{k-1}} a_k \right) *^{n-1}\left(y_1*^{\epsilon_{k-1}}a_k\right) \\
&=\left(x*^{-\epsilon_{k-1}} a_k *^2y_1 *^{\epsilon_{k-1}} a_k\right) *^{n-2}\left(y_1*^{\epsilon_{k-1}}a_k\right) \\
&=\left(\left(x*^{-\epsilon_{k-1}} a_k\right) *^ny_1\right) *^{\epsilon_{k-1}} a_k \\
&=\left(x*^{-\epsilon_{k-1}} a_k\right) *^{\epsilon_{k-1}} a_k, \quad \textrm{by induction hypothesis}\\
&=x.
\end{align*}
Thus, $X$ is a $n$-quandle, proving our claim. Finally, it follows from the universal property that $(Q)_n\cong X$.
\end{proof}

\begin{corollary}
If the $n$-quandle of a quandle is finite for some integer $n \ge 2$, then its $d$-quandle is also finite for each divisor $d$ of $n$.
\end{corollary}

The presentation of the enveloping group of a quandle can be reduced substantially  \cite[Theorem 5.1.7]{MR2634013}.

\begin{corollary}\label{presen_n_quandle_env}
Let $Q=\langle S \mid R\rangle$ be a quandle presentation. Then 
$$\Env((Q)_n)=\langle \tilde{S} \mid \tilde{R}\cup \tilde{T} \rangle \cong \Env(Q)/\llangle \tilde{T} \rrangle,$$ 
where $\tilde{S}= \{ e_x \mid x \in S \}$, $\tilde{T}=\{e_x^n e_ye_x^{-n}e_y^{-1} \mid x,y\in S\}$ and $\tilde{R}$ consists of relations of $R$ with each $x \ast y$ replaced by $e_{y} e_x e_y^{-1}$.
\end{corollary}

Given a quandle $Q$, define a map $\nu:Q\to \mathbb{N}\cup\{0\}$ by setting 
$$\nu(x)=\begin{cases} 
\text{order of~}S_x, &~\textrm{if order of}~S_x~ \textrm{is finite},\\
0, &~\textrm{if order of}~S_x~ \textrm{is infinite}.\\
\end{cases}$$
Consider the subgroup $Z(Q)= \langle e_x^{\nu(x)} \mid x\in Q \rangle $ of $\Env(Q)$. It follows from the construction of $Z(Q)$ that it is a central subgroup of $\Env(Q)$ \cite[Lemma 2.1]{Akita2022}. Let $F(Q)=\Env(Q)/Z(Q)$ be the corresponding quotient group. It is known that if $Q$ is a finite quandle, then $F(Q)$ is a finite group \cite[Proposition 3.1]{Akita2022}.

\begin{theorem}\label{finiteness of g_n}
Let $G$ be a group generated by a set $S$ and $G_n=G/\llangle s^n \mid s\in S \rrangle$. Then there exists a surjective group homomorphism $F\big((\dq(S^G))_n \big)\to G_n$. In particular, if $(\dq(S^G))_n$ is finite, then $G_n$ is finite.
\end{theorem}

\begin{proof}
It follows from \cite[Theorem 3.6]{Dhanwani-Raundal-Singh-2021} that the map $\Phi:\Env(\dq(S^G))\to G$, given by $\Phi(e_x)=x$, is a  surjective homomorphism.  In view of Corollary \ref{presen_n_quandle_env}, we have a surjective homomorphism $\tilde{\Phi}:\Env((\dq(S^G))_n)\to H_n$, where $H_n=G/\llangle x^nyx^{-n}y^{-1} \mid x,y\in S \rrangle$. This, in turn, gives a surjective homomorphism $$\bar{\Phi}:F((\dq(S^G))_n)\to H_n/\llangle x^n \mid x\in S \rrangle.$$ The assertion now follows from the fact that $H_n/\llangle x^n \mid x\in S \rrangle \cong G_n$.
\end{proof}
\medskip

\section{Involutory quotients of Artin quandles}\label{inv quotients of Artin quandle}

Recall that an Artin group ${\bf A}$ is a group with a presentation $${\bf A}=\Bigg\{{\bf s}_1, {\bf s}_2,\ldots, {\bf s}_n \mid ({\bf s}_i {\bf s}_j)_{m_{ij}}=({\bf s}_j {\bf s}_i)_{m_{ij}}, ~\textrm{where}~m_{ij}\in \{2,3,\ldots\} \cup \{ \infty\}\Bigg\},$$ where $({\bf s}_i {\bf s}_j)_{m_{ij}}$ is the word ${\bf s}_i {\bf s}_j {\bf s}_i {\bf s}_j {\bf s}_i \ldots$ of length $m_{ij}$ if $m_{ij} < \infty$ and there is no relation of type $({\bf s}_i {\bf s}_j)_{m_{ij}}=({\bf s}_j {\bf s}_i)_{m_{ij}}$ if $m_{ij}=\infty$. We set ${\bf S}= \{{\bf s}_1, {\bf s}_2,\ldots, {\bf s}_n \}$.  The corresponding Coxeter group $W$ is the quotient of ${\bf A}$ by imposing additional relations ${\bf s}_i^2=1$ for all ${\bf s}_i \in {\bf S}$. 
\par
To distinguish the presentation of  ${\bf A}$ from that of its corresponding Coxeter group $W$, we write the presentation of $W$ without using bold letters. More precisely,
$$W=\Bigg\{s_1,s_2,\ldots,s_n \mid s_i^2=1, \quad (s_is_j)_{m_{ij}}=(s_js_i)_{m_{ij}}, ~\textrm{where}~m_{ij}\in \{2,3,\ldots\} \cup \{ \infty\}\Bigg\}.$$ 
Setting $S= \{ s_1, s_2,\ldots, s_n \}$, the pair $(W, S)$ is referred as a Coxeter system and elements of $W$ which are conjugates of elements of $S$ are called reflections.  Following \cite{Dhanwani-Raundal-Singh-2021}, Dehn quandles of Artin and Coxeter groups with respect to their standard generating sets will be referred as Artin and Coxeter quandles, respectively. In this section, we prove that the involutory quandle of an Artin quandle is the corresponding Coxeter quandle.
\par

To prove the result, we first understand centralisers of Coxeter generators, for which we follow \cite{MR3084328}. Let $(W,S)$ be a Coxeter system and $\triangle$ the labelled graph with vertex set $S$ such that there is an edge between $s$ and $s'$ with label $m_{ss'}$ whenever $m_{ss'} < \infty$. Let $\triangle^{\rm{odd}}$ be the subgraph of $\triangle$ consisting of only odd labelled edges. It is easy to see that connected components of $\triangle^{\rm{odd}}$ correspond to conjugacy classes of reflections in $W$. For each $s \in S$, let $\triangle_{s}^{\rm{odd}}$ denote the connected component of $\triangle^{\rm{odd}}$ containing $s$.
\par
Let $(W, S)$ be a Coxeter system and $\gamma = (t_0, t_1, \ldots,t_n)$ an edge-path in $\triangle^{\rm{odd}}$ with the edge joining $t_{i-1}$ and $t_i$ labelled $2l_i + 1$. We set 
$$ p_\gamma:=(t_1t_0)^{l_1}(t_2t_1)^{l_2}\cdots(t_nt_{n-1})^{l_n}$$
if $n>0$ and $p_{\gamma}=1$ if $n=0$. If $u$ is a vertex of $\triangle$ such that there is an edge joining $u$ and $t_n$ with even label $2\lambda$, then we define
$$ r_{\gamma,u}:=p_{\gamma}~(ut_n)^{\lambda-1}u~ p^{-1}_{\gamma}.$$
Henceforth, the notation $r_{\gamma,u}$ means that the vertex $u$ is joined to the end point of $\gamma$ by an even labelled edge. We need the following results  \cite[Corollary 6 and Corollary 8]{MR3084328}. See also \cite{MR1396145, MR1688445, MR0576184}.

\begin{theorem}\label{Centralizer of parabolic element in Coxeter}
Let $(W,S)$ be a Coxeter system, $s\in S$ and $\C_{W}(s)$ the centraliser of $s$ in $W$. 
\begin{enumerate}
\item Let $W_{\Omega}$ be the subgroup of $\C_{W}(s)$ generated by all the reflections it contains except $s$. Let $\Gamma_{\Omega}$ be the subgroup of $W$ generated by elements $p_\gamma$, where $\gamma$ is an edge-loop in $\triangle^{\rm{odd}}$ based at $s$. Then $\C_{W}(s) = \langle s \rangle \times(W_{\Omega}\rtimes\Gamma_{\Omega})$.
\item  Let $Z$ be a set of edge-loops in $\triangle^{\rm{odd}}_s$ generating the fundamental group $\pi_1(\triangle^{\rm{odd}},s)$. Then $\{p_\gamma \mid \gamma \in Z\}$ generate $\Gamma_\Omega$. 
\item Let $Y$ be the set consisting of one edge-path $\delta_t$  in  $\triangle^{odd}_{s}$ from $s$ to $t$ for each vertex $t\in\triangle^{odd}_{s}$, and let $X$ be the set of vertices $u$ of $\triangle$ such that there is an even labelled edge joining  $u$ to the end point of some $\delta_t$. Then $\{p_\gamma \mid \gamma \in Z\}$ and $\{r_{\delta_t,u} \mid \delta_t \in Y,~u \in X \}$ together generate $W_{\Omega}\rtimes \Gamma_{\Omega}$.
\end{enumerate}
\end{theorem}

Let $\nu : {\bf A} \to W$ be the natural surjection given by $\nu({\bf s})=s$. Clearly, $\ker(\nu)$ is normally generated by $\{{\bf s}^2 \mid {\bf s} \in {\bf S}\}$. 

\begin{lemma}\label{Cent_Artin_to_Cent_Coxeter}
The restriction map $\C_{{\bf A}}({\bf s})\to \C_{W}(s)$ is surjective for each ${\bf s}\in {\bf S}$.
\end{lemma}

\begin{proof}
Let ${\bf s}\in {\bf S}$ be fixed such that  $\nu({\bf s})=s$. Let  $\gamma = (t_0, t_1, \ldots,t_n)$ be an edge-path in $\triangle^{\rm{odd}}$ with the edge joining $t_{i-1}$ and $t_i$ labelled $2l_i + 1$ and let $p_{\gamma}=(t_1t_0)^{l_1}(t_2t_1)^{l_2}\cdots(t_nt_{n-1})^{l_n}$ the corresponding element. Let ${\bf p}_{\gamma}=({\bf t}_1{\bf t}_0)^{l_1}({\bf t}_2 {\bf t}_1)^{l_2}\cdots({\bf t}_n {\bf t}_{n-1})^{l_n}$ denote the lift of $p_{\gamma}$ in ${\bf A}$. If we set $\bar{{\bf p}}_{\gamma}=({\bf t}_{n-1} {\bf t}_n)^{l_n}({\bf t}_{n-2} {\bf t}_{n-1})^{l_{n-1}}\cdots({\bf t}_0 {\bf t}_1)^{l_1}$, then it follows that $\nu(\bar{{\bf p}}_\gamma)=p_{\gamma}^{-1}$. It is easy to see that the  identities 
$${\bf p}_{\gamma} {\bf t}_n = {\bf t}_0 {\bf p}_{\gamma} \quad \textrm{and} \quad \bar{\bf p}_{\gamma} {\bf t}_0 = {\bf t}_n \bar{\bf p}_{\gamma}$$
hold in ${\bf A}$. Thus, if $\gamma$ is an edge loop based at $s$, then both ${\bf p}_{\gamma}$ and $\bar{\bf p}_{\gamma}$ commute with ${\bf s}$. Similarly, let $\delta_t$ be a edge-path in $\triangle_{s}^{odd}$ from $s$ to $t$ and $u$ a vertex of $\triangle$ that is joined to the end point of $\delta_t$ with an edge labelled $2\lambda$. If we set ${\bf r}_{\delta_t,u}={\bf p}_{\delta_t}~({\bf u} {\bf t})^{\lambda-1} {\bf u}~ \bar{{\bf p}}_{\delta_t}$, then we can see that 
\begin{eqnarray*}
{\bf r}_{\delta_t,u} {\bf s} &=& {\bf p}_{\delta_t}~({\bf u} {\bf t})^{\lambda-1} {\bf u}~ \bar{{\bf p}}_{\delta_t} {\bf s}\\
&=& {\bf p}_{\delta_t}~( {\bf u} {\bf t})^{\lambda-1} {\bf u}~{\bf t}~ \bar{{\bf p}}_{\delta_t}\\
&=& {\bf p}_{\delta_t}~{\bf t}~({\bf u} {\bf t})^{\lambda-1} {\bf u}~ \bar{{\bf p}}_{\delta_t}\\
&=& {\bf s} {\bf p}_{\delta_t}~({\bf u} {\bf t})^{\lambda-1} {\bf u}~ \bar{{\bf p}}_{\delta_t}\\
&=& {\bf s} {\bf r}_{\delta_t,u},
\end{eqnarray*}
and hence ${\bf r}_{\delta_t,u}$ commutes with ${\bf s}$.  By Theorem \ref{Centralizer of parabolic element in Coxeter}, the group $\C_{W}(s)$ is generated by $\{ s, p_\gamma, r_{\delta_t,u}\}$. Since each generator has a pre-image under $\nu$ that lies in $\C_{{\bf A}}({\bf s})$, the assertion follows.
\end{proof}

We now prove the main result of this section.

\begin{theorem}\label{Artin to Coxeter quandle}
The involutory quandle of an Artin quandle is the corresponding Coxeter quandle, i.e, $(\dq({\bf S}^{\bf A}))_2 \cong \dq(S^W).$
\end{theorem}

\begin{proof}
The surjection $\nu: {\bf A}\to W$ induces a surjective quandle homomorphism $\tilde{\nu}: \dq({\bf S}^{\bf A}) \to \dq(S^{W})$. Since $(\dq({\bf S}^{\bf A}))_2$ is involutory, $\tilde{\nu}$ descends to a surjective quandle homomorphism $$\overline{\nu}:(\dq({\bf S}^{\bf A}))_2\to \dq(S^{W}).$$ If $\pi:\dq({\bf S}^{\bf A})\to (\dq({\bf S}^{\bf A}))_2$ is the natural surjection, then $\overline{\nu} ~\pi= \tilde{\nu}$. We claim that $\overline{\nu}$ is an isomorphism. Suppose that $\tilde{\nu}({\bf x})=\tilde{\nu}({\bf y})$ for some ${\bf x},{\bf y}\in \dq({\bf S}^{\bf A})$. Since $\tilde{\nu}$ preserves orbits (under action of the inner automorphism group), it follows that ${\bf x}$ and ${\bf y}$ must be in the same orbit. Thus, we can write  ${\bf x}= {\bf s}_{i_k}^{\epsilon_k} \cdots {\bf s}_{i_1}^{\epsilon_1} {\bf s} {\bf s}_{i_1}^{-\epsilon_1} \cdots  {\bf s}_{i_k}^{-\epsilon_k}$ and ${\bf y}={\bf s}_{j_r}^{\delta_r} \cdots {\bf s}_{j_1}^{\delta_1} {\bf s} {\bf s}_{j_1}^{-\delta_1} \cdots  {\bf s}_{j_r}^{-\delta_r}$ for some ${\bf s}, {\bf s}_{i_t}, {\bf s}_{j_t}\in {\bf S}$ and $\epsilon_t, \delta_t \in \{ 1, -1 \}$. Now $\tilde{\nu}({\bf x})=\tilde{\nu}({\bf y})$  implies that $\nu({\bf s}_{j_1}^{-\delta_1} \cdots {\bf s}_{j_r}^{-\delta_r} {\bf s}_{i_k}^{\epsilon_k} \cdots {\bf s}_{i_1}^{\epsilon_1})$ commutes with $\nu({\bf s})=s$. In view of Lemma \ref{Cent_Artin_to_Cent_Coxeter}, the short exact sequence $1\to \text{ker}(\nu) \to {\bf A} \stackrel{\nu}{\to} W\to 1$ induces a short exact sequence    
\begin{equation}\label{congruence exact sequence Artin}
1\to  \text{ker}(\nu)\cap\C_{{\bf A}}({\bf s})\to \C_{{\bf A}}({\bf s})\to \C_{W}(s)\to 1.
\end{equation}
Thus, we can write ${\bf s}_{j_1}^{-\delta_1} \cdots {\bf s}_{j_r}^{-\delta_r} {\bf s}_{i_k}^{\epsilon_k} \cdots {\bf s}_{i_1}^{\epsilon_1}={\bf u} {\bf v}$ for some ${\bf u} \in\ker(\nu)=\llangle {\bf s}^2 \mid {\bf s}\in {\bf S} \rrangle$ and some ${\bf v} \in  \C_{{\bf A}}({\bf s})$.  Thus, we have
\begin{eqnarray*}
&& \pi( {\bf s}_{j_1}^{-\delta_1} \cdots {\bf s}_{j_r}^{-\delta_r} {\bf s}_{i_k}^{\epsilon_k} \cdots {\bf s}_{i_1}^{\epsilon_1} ~{\bf s}~ {\bf s}_{i_1}^{-\epsilon_1} \cdots {\bf s}_{i_k}^{-\epsilon_k} {\bf s}_{j_r}^{\delta_r} \cdots {\bf s}_{j_1}^{\delta_1})\\
&=&\pi({\bf u} {\bf v} {\bf s} {\bf v}^{-1} {\bf u}^{-1})\\
&=&\pi({\bf u} {\bf s}  {\bf u}^{-1}), \quad \text{since ${\bf v} \in  \C_{{\bf A}}({\bf s})$}\\
&=& \pi({\bf s}), \quad \text{since ${\bf u} \in \ker(\nu)$},
\end{eqnarray*}
which shows that $\pi({\bf x})=\pi({\bf y})$, and hence $\overline{\nu}$ is an isomorphism. 
\end{proof}

As a consequence of propositions \ref{lem:presen_n_quandle} and \ref{Artin to Coxeter quandle}, we obtain a presentation for the enveloping group of the Coxeter quandle, which extends \cite[Proposition 3.3]{MR4175808}.

\begin{corollary}
Let $(W, S)$ be a Coxeter system. Then $\Env(\dq(S^{W})) \cong {\bf A}/{\bf N}$, where ${\bf A}$ is the corresponding Artin group and ${\bf N}=\llangle ~{\bf s}_i^2 {\bf s}_j {\bf s}_i^{-2} {\bf s}_j^{-1} \mid {\bf s}_i, {\bf s}_j\in {\bf S}\rrangle$.
\end{corollary}

Our next result gives the size of the smallest non-trivial quandle quotient of the Dehn quandle of the braid group with respect to its standard generating set. The proof is essentially rephrasing \cite[Lemma 8]{kolay21} in the language of quandles.

\begin{proposition}\label{bound of braid quandle image}
Let $n\geq 5$ and $f:\dq(S^{B_n}) \to Q$ be a surjective quandle homomorphism onto a quandle with at least two elements. Then $|Q|\geq \frac{n(n-1)}{2}$, and the bound is sharp. 
\end{proposition}

\begin{proof}
We view the braid group $B_n$ as the mapping class group of the disc $\mathbb{D}^2$ with $n$ marked points. Consider the set  $\{c_{i, j} \mid 1 \le i < j \le n \}$ of $n(n-1)/2$ arcs on $\mathbb{D}^2$ as shown in Figure \ref{figure:essen_arc} for $n=6$.
\begin{figure}
\includegraphics[width=0.33\textwidth]{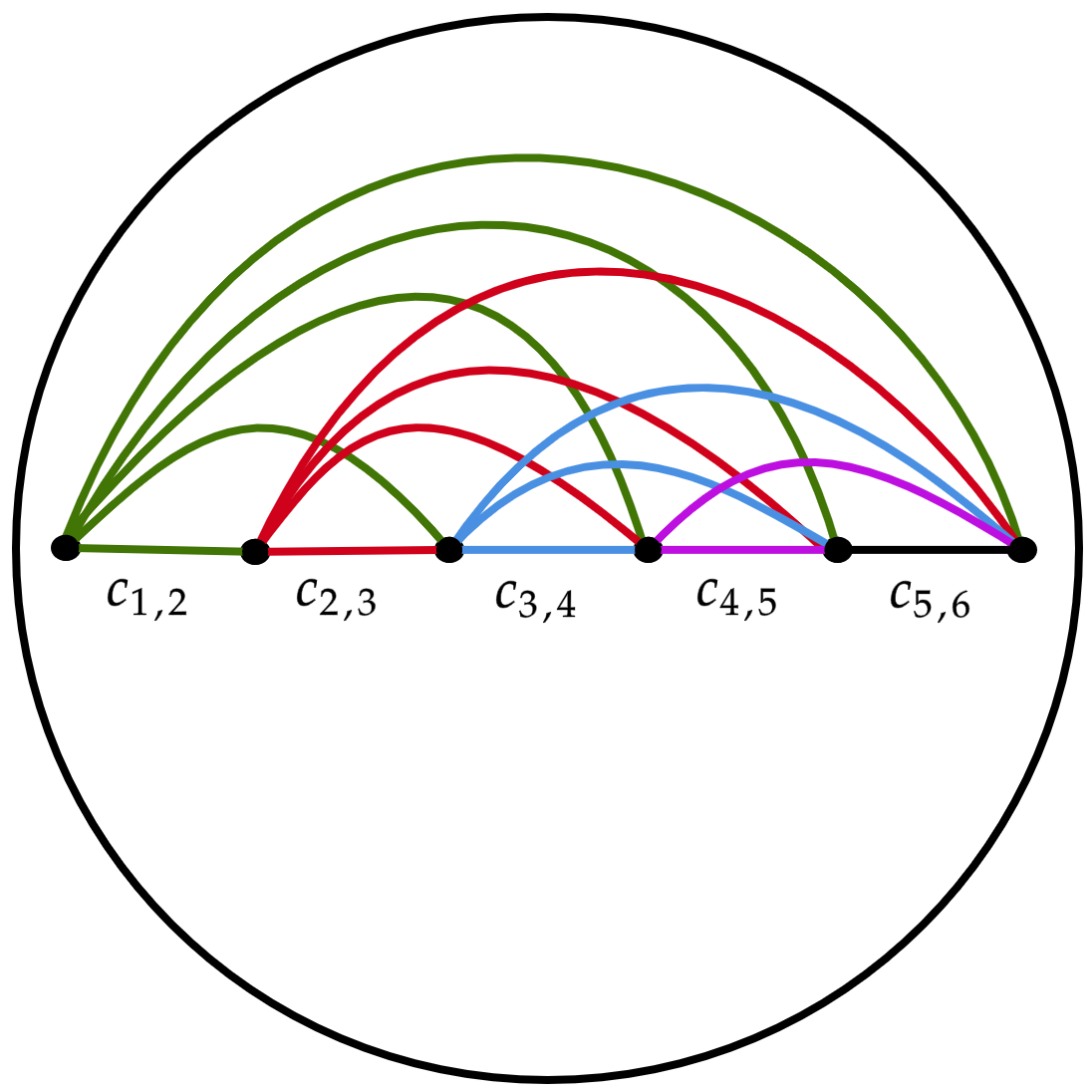}
\caption{Arcs $c_{i, j}$ on the disc for $n=6$.}
\label{figure:essen_arc}
\end{figure}
 Let $\sigma_{i, j}$ denote the isotopy class of the (anti-clockwise) half Dehn twist along the arc $c_{i, j}$ and let $X= \{\sigma_{i, j} \mid 1 \le i < j \le n  \}$. Then $B_n$ is generated by the subset $S=\{\sigma_{1, 2}, \sigma_{2, 3}, \ldots, \sigma_{n-1, n} \}$ of $X$. Note that all the elements of $X$ are conjugate to each other in $B_n$, and hence lie in $\dq(S^{B_n})$. Let $f:\dq(S^{B_n}) \to Q$ be a surjective quandle homomorphism onto a quandle with at least two elements. Then we have the following cases:
\par
Case (1): Suppose that $f$ is not injective on $S$, that is,  $f(\sigma_{i, i+1})=f(\sigma_{j, j+1})$ for some $i\neq j$. Since $n \ge 5$, by construction of $X$, we can choose an arc $c_{k, k+1}$ such that it has exactly one end point in common with precisely one of $c_{i, i+1}$ or $c_{j, j+1}$ and no end point in common with the other. Without loss of generality, we can assume that $c_{k, k+1}$ and $c_{j, j+1}$ have one end point in common, while $c_{k, k+1}$ and $c_{i, i+1}$ are disjoint. Then, the half Dehn twists satisfy the relations 
$$\sigma_{k, k+1} \sigma_{j, j+1}  \sigma_{k, k+1} = \sigma_{j, j+1}  \sigma_{k, k+1}\sigma_{j, j+1}  \quad \textrm{and} \quad \sigma_{k, k+1} \sigma_{i, i+1}= \sigma_{i, i+1}\sigma_{k, k+1} .$$
Applying $f$ to the preceding relations and using the equality $f(\sigma_{i, i+1})=f(\sigma_{j, j+1})$, we obtain $f(\sigma_{k, k+1}) =f( \sigma_{i, i+1})$. Proceeding in this manner, we show that $f$ is constant on $S$. Since $S$ also generates $\dq(S^{B_n})$ as a quandle \cite[Proposition 3.2]{Dhanwani-Raundal-Singh-2021}, it follows that $f$ is constant on $\dq(S^{B_n})$, which is a contradiction. Hence, $f$ is injective on $S$.
\par
Case (2): Suppose that $f$ is not injective on $X$. There are two possibilities here. 
\par
Subcase (2a): Suppose that $f(\sigma_{i, i+1})=f(\sigma_{k, l})$ where $\{i, i+1 \} \ne \{k, l \}$ as sets. Since $n \ge 5$, we can choose an arc $c_{r,r+1}$ such that 
it has exactly one end point in common with precisely one of $c_{i, i+1}$ or $c_{k, l}$ and no end point in common with the other. If $c_{r, r+1}$ and $c_{i, i+1}$ have one end point in common, then as in Case (1) we obtain  $f(\sigma_{i, i+1}) =f( \sigma_{r, r+1})$. Similarly, if $c_{r, r+1}$ and $c_{k, l}$ have one end point in common, then as in Case (1) we get $f(\sigma_{i, i+1}) =f( \sigma_{r, r+1})$. Again, proceeding as in Case (1) leads to a contradiction.
\par
Subcase (2b): Suppose that $f(\sigma_{i, j})=f(\sigma_{k, l})$ where $\{i, j\} \ne \{k, l\}$. Since $n \ge 5$, we can choose an arc $c_{r,r+1}$ such that it has exactly one end point in common with precisely one of $c_{i, j}$ or $c_{k, l}$ and no end point in common with the other. Proceeding as before, we deduce that $f(\sigma_{i, j})=f(\sigma_{k, l})= f(\sigma_{r,r+1})$, which is Subcase (2a).
\par
Hence, it follows that the map $f$ must be injective on the set $X$, and therefore $|Q|\geq n(n-1)/2$. For sharpness of the bound, consider the surjective group homomorphism $B_n \to \Sigma_n$, where  $\Sigma_n$ is the symmetric group on $n$ symbols. The  group homomorphism induces a surjective quandle homomorphism $\dq(S^{B_n}) \to \dq(T^{\Sigma_n})$, where $T=\{(i, i+1) \mid 1 \le i \le n-1 \}$ is the image of $S$. Since $|\dq(T^{\Sigma_n})|= n(n-1)/2$, the assertion follows. 
\end{proof}

\begin{remark}
Note that, by Theorem \ref{Artin to Coxeter quandle},  $\dq(T^{\Sigma_n})$ is precisely the 2-quandle of  $\dq(S^{B_n})$. It is intriguing to know smallest quandle quotients of general Artin quandles.
\end{remark}
\medskip

\section{Finiteness of $n$-quandles of Dehn quandles of surfaces}\label{inv quotients of Dehn quandle}
In this main section, we consider finiteness of $n$-quandles of Dehn quandles of closed orientable surfaces. For each genus, with some exception, we determine all values of $n$ for which the $n$-quandle of the Dehn quandle of the surface $S_g$ is finite. We also determine the explicit structures of their 2-quandles.

\subsection{2 and 3-quandles of Dehn quandles of surfaces} 
Let  $S_g$ be a closed oriented surface of genus $g \ge 1$ and $$\Psi : \mathcal{M}_g \to \text{Sp}(2g, \mathbb{Z})$$ the symplectic representation of the mapping class group $\mathcal{M}_g$ of $S_g$. 
For an integer $n \ge 2$, let $\psi_n:\text{Sp}(2g,\mathbb{Z}) \to \text{Sp}(2g,\mathbb{Z}_n)$ be the mod $n$ reduction homomorphism and $$\Psi_n:= \psi_n \Psi :\mathcal{M}_g \to \text{Sp}(2g, \mathbb{Z}_n)$$ the mod $n$ reduction of the symplectic representation. Given a simple closed curve $a$ in $S_g$, let $\tau_a$ denote the (right) Dehn twists about $a$. Consider the surface $S_g$ with curves $a_i, b_i, c_i$ as in Figure \ref{figure2}.

\begin{figure}
\includegraphics[width=0.8\textwidth]{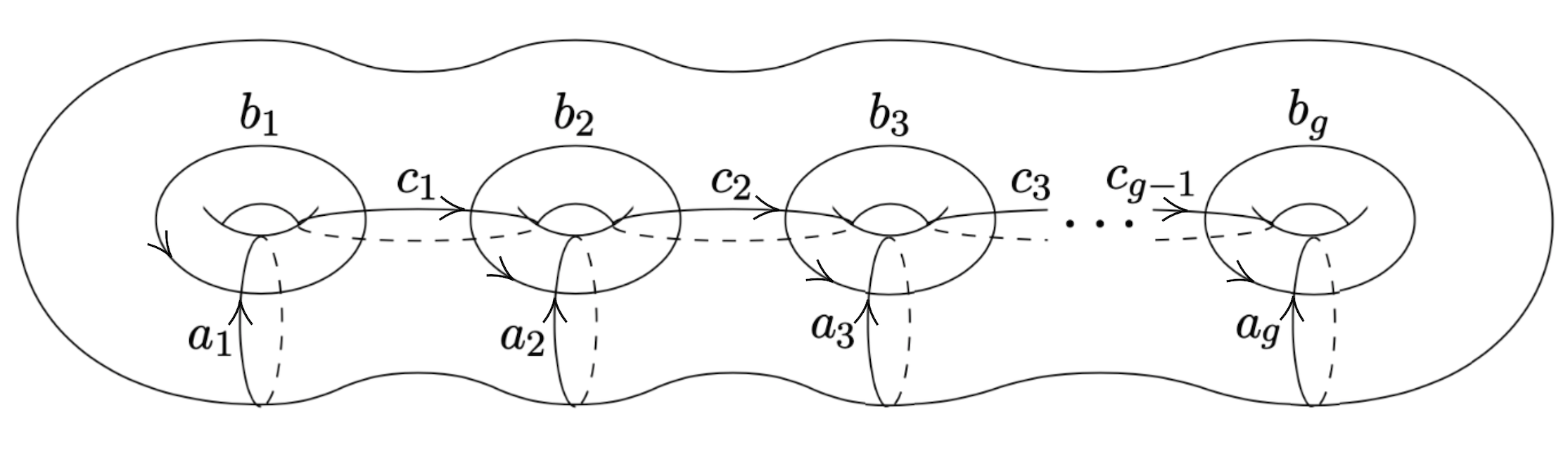}
\caption{Surface of genus $g$ with curves $a_i$, $b_i$ and $c_i$.}
\label{figure2}
\end{figure}

\begin{lemma}\label{com-a1-in-modp}
Let $p$ be a prime and $g \ge 1$. Then, an element $A\in \mathrm{Sp}(2g,\mathbb{Z}_p)$ lies in the centraliser of $\Psi_p(\tau_{a_1})$ if and only if $A$ has first column $(a,0,0,\ldots,0)$ and second row $(0,a,0,0,\ldots,0)$, where $a=\pm 1 \pmod p$.
\end{lemma}

\begin{proof}
Let $E_{i,j}$ denote the elementary matrix with $(i,j)$-entry as $1$ and each other entry as $0$. Observe that $\Psi_p(\tau_{a_1})=I_{2g}-E_{1,2}$. If $A\in \mathrm{Sp}(2g,\mathbb{Z}_p)$ commutes with $\Psi_p(\tau_{a_1})$, then it follows from elementary linear algebra that $A$ has first column $(a,0,0,\ldots,0)$ and second row $(0,a,0,0,\ldots,0)$ for some $a\in \mathbb{Z}_p^\times$. Since $A^{t}JA=J$, where $J$ represents the symplectic form, we obtain $a^2=1 \pmod p$. Since  $a^2=1 \pmod p$ if and only if $ a=\pm 1\pmod p$, the assertion follows. Conversely, any symplectic matrix of this form commutes with $\Psi_p(\tau_{a_1})$.
\end{proof}

\begin{lemma}\label{gen-com-a1-in-modp}
Let $p$ be a prime  and $g \ge 1$. Then the centraliser of  $\Psi_p(\tau_{a_1})$ in $ \mathrm{Sp}(2g,\mathbb{Z}_p)$ is generated by $\Psi_p(S)\cup\{-I_{2g}\}$, where $$S=\{ \tau_{a_1},\tau_{a_2},\ldots,\tau_{a_g},\tau_{b_2},\ldots,\tau_{b_g},\tau_{c_1},\tau_{c_2},\ldots,\tau_{c_{g-1}} \}.$$
\end{lemma}

\begin{proof}
For $g=1$, a direct computation shows that the centraliser of  $\Psi_p(\tau_{a_1})$ in $ \mathrm{Sp}(2,\mathbb{Z}_p)$ is generated by $\Psi_p(\tau_{a_1})\cup\{-I_{2g}\}$.  We assume that $g \ge 2$. Clearly, $\Psi_p(S)\cup\{-I_{2g}\}$ is a subset of the centraliser of  $\Psi_p(\tau_{a_1})$ in $ \mathrm{Sp}(2g,\mathbb{Z}_p)$. If $A=(A_{i,j})$ lies in the centraliser of $\Psi_p(\tau_{a_1})$ in $\mathrm{Sp}(2g,\mathbb{Z}_p)$, then Lemma \ref{com-a1-in-modp} implies that $A_{1,1}= \pm 1$. Multiplying $A$ by $-I_{2g}$, we can assume that $A_{1,1}=1$.  It remains to prove such an $A$ can be written as a product of elements from $\Psi_p(S)$. Consider the curves $d_1,d_2,\ldots,d_{g-1}$ as shown in Figure \ref{curves}.
\begin{figure}
\includegraphics[width=0.7\textwidth]{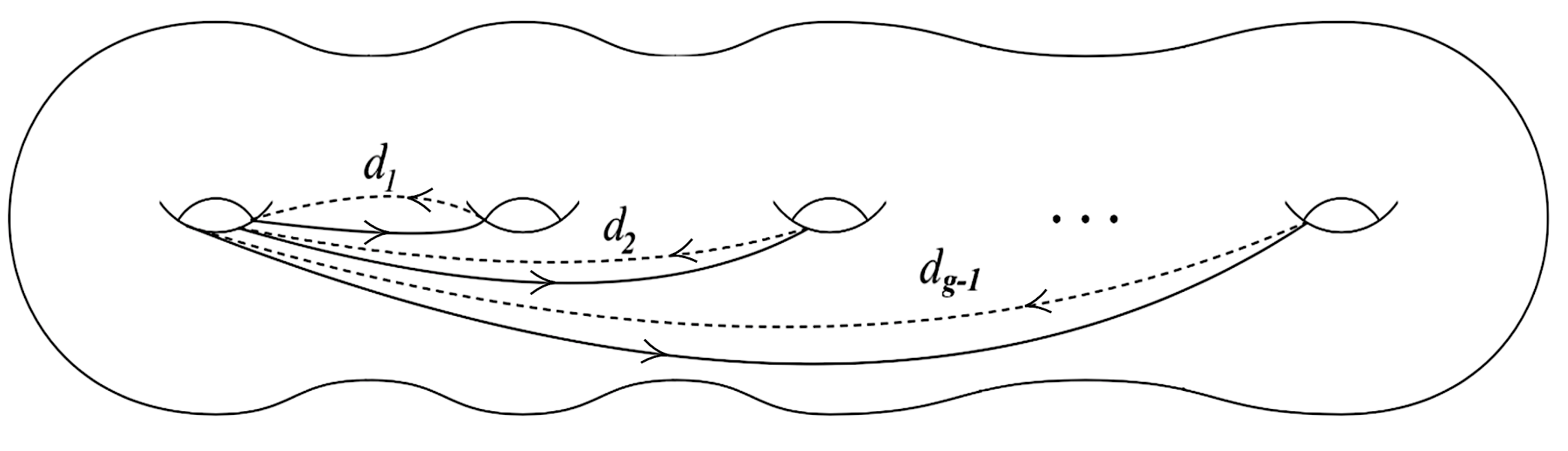}
\caption{Surface of genus $g$ with curves $d_i$.}
\label{curves}
\end{figure}
For each $1 \le i \le g-1$, let $M_i:=\Psi_p(\tau_{d_i}\tau_{a_1}^{-1}\tau_{a_{i+1}}^{-1})$. Then, we can see that $$M_i=I_{2g}+E_{2i+1,\,2}+E_{1,\,2i+2}.$$ 
Let $N_i:= \Psi_p(\tau_{a_{i+1}}\tau_{b_{i+1}}\tau_{a_{i+1}})^{-1} M_i ~\Psi_p(\tau_{a_{i+1}}\tau_{b_{i+1}}\tau_{a_{i+1}})$. Then, a direct computation shows that $$N_i=I_{2g}-E_{2i+2,\,2}+E_{1,\,2i+1}.$$ 
Now, consider the matrix
$$B:=\Bigg(\prod_{i=1}^{g-1}N_{g-i}^{A_{2(g-i)+2,2}}M_{g-i}^{-A_{2(g-i)+1,2}} \Bigg)A,$$ 
where a matrix raised to the power by the entry $A_{i,j}$ represents the same matrix raised to the power by a representative of $A_{i,j}$ in the subset $\{ 0, 1, 2, \ldots, p-1 \}$ of integers. By the choice of matrices $M_i$ and $N_i$, it follows that the second column of $B$ is $(B_{1,2},1,0,0,\ldots,0)$. Next, take 
$$C:=\Psi_p(\tau_{a_1})^{B_{1,2}} B.$$ It is easy to see that the first two columns of $C$ are $(1, 0, 0, \ldots, 0)$ and $(0, 1, 0, \ldots, 0)$, respectively. Since $C$ is symplectic, it has the form
$$C=\begin{bmatrix}
I_2 & O_{2,2g-2}\\
O_{2g-2,2} & D
\end{bmatrix},
$$ 
where $O_{m,n}$ is the $m\times n$ null matrix  and $D \in \text{Sp}(2g-2,\mathbb{Z}_p)$. Note that  $\text{Sp}(2g-2,\mathbb{Z}_p)$ is generated by $\Psi_2(\{\tau_{a_2},\ldots,\tau_{a_g},\tau_{b_2},\ldots,\tau_{b_g},\tau_{c_2},\ldots,\tau_{c_{g-1}} \}) $. We claim that each $\tau_{d_i}$ can be written as product of elements from $S$, which would complete the proof. Since  $d_1=c_1$, the claim follows that $g=2$.

\par
Let $\mcg_g(\overrightarrow{\alpha})$ be the subgroup of $\mcg_g$ consisting of all elements that preserve a non-separating simple closed curve $\alpha$ with orientation. Let $\mcg_{g,\alpha}$ be the mapping class group of the surface $S_{g, \alpha}$ (which is homeomorphic to a surface of genus $g-1$ and two boundary components) obtained by cutting the surface $S_g$ along the curve $\alpha$ and then taking its closure. Let $\{\delta_1,\delta_2\}$ be the two boundary components of $S_{g, \alpha}$. Then, by \cite[Lemma 1.20]{Putman}, there is a short exact sequence of groups 
$$1\to \langle \tau_{\delta_1}\tau_{\delta_2}^{-1} \rangle \to \mcg_{g,\alpha} \to \mcg_g(\overrightarrow{\alpha})\to 1,$$ where $\langle \tau_{\delta_1}\tau_{\delta_2}^{-1} \rangle \cong \mathbb{Z}$.
Applying this result to the oriented curve $a_1$, and using the generating set for $\mcg_{g, a_1}$ (see \cite[Corollary 4.16 and Figure 4.10]{Farb-Margalit2012}), we obtain 
$$\mcg_g(\overrightarrow{a_1})= \langle \tau_{a_1},\tau_{a_2},\ldots,\tau_{a_g},\tau_{b_2},\ldots,\tau_{b_g},\tau_{c_1},\tau_{c_2},\ldots,\tau_{c_{g-1}} \rangle.$$
Since each $\tau_{d_i}$ preserve the curve $a_1$ with orientation, the claim follows.
\end{proof}

Using Lemma \ref{gen-com-a1-in-modp} and choosing the hyperelliptic involution as a pre-image of $-I_{2g}$, the short exact sequence 
$$1\to \ker(\Psi_p) \longrightarrow \mcg_g \stackrel{\Psi_p}{\longrightarrow} \text{Sp}(2g,\mathbb{Z}_p)\to 1$$
induces the following short exact sequence at the level of centralisers.

\begin{corollary}\label{centraliser short exact mcg}
For each prime $p$ and each $g \ge 1$, there is a short exact sequence of groups
\begin{equation}\label{congruence exact sequence}
1\to \ker(\Psi_p) \cap\C_{\mcg_g}(\tau_{a_1})\to \C_{\mcg_g}(\tau_{a_1})\to \C_{\mathrm{Sp}(2g,\mathbb{Z}_p)}(\Psi_p(\tau_{a_1}))\to 1.
\end{equation}
\end{corollary}

We briefly recall from \cite[Section 7]{Dhanwani-Raundal-Singh-2021} the construction of the  projective primitive homological quandle  $\mathcal{P}_{g,n}$ of the surface $S_g$, where $g \ge 1$ and $n \ge 2$. Let $\mathcal{P}'_g$ denote the set of all primitive elements in $\mathrm{H}_1(S_g, \mathbb{Z})$ and $\mathcal{P}'_{g,n}$ denote the set of all primitive elements in $\mathrm{H}_1(S_g,\mathbb{Z}_n)$. The algebraic intersection number $\hat{i}(-, -)$ gives a skew-symmetric (in fact, symplectic) bilinear form on the $\mathbb{Z}$-module $\mathrm{H}_1(S_g,\mathbb{Z})$.  For $x, y \in \mathcal{P}'_g$, the binary operation $$x*y :=x + \hat{i}(x,y) y$$ gives a quandle structure on $\mathcal{P}'_g$. Similarly, reduction modulo $n$ defines a quandle structure on $\mathcal{P}'_{g,n}$.
\par
Let $\mathcal{P}_g:=\mathcal{P}'_g/\mathbb{Z}_2$ and $\mathcal{P}_{g,n}:=\mathcal{P}'_{g,n}/\mathbb{Z}_2$ be quotients under the natural involutory action of $\mathbb{Z}_2=\{1,-1\}$. It is clear that quandle structures descends to quandle structures on $\mathcal{P}_g$ and $\mathcal{P}_{g,n}$. Further, reduction modulo $n$ gives a surjective quandle homomorphism $ \mathcal{P}'_g \to \mathcal{P}'_{g,n}$, which further induces a surjective quandle homomorphism $\mathcal{P}_g \to \mathcal{P}_{g,n}$. The quandles $\mathcal{P}_g$ and  $\mathcal{P}_{g,n}$ are called {\it projective primitive homological quandles} of $S_g$. Note that  $\mathcal{P}_{g,2}$ is an involutory quandle of order $2^{2g}-1$.
\par

Recall from \cite[Proposition 6.2]{Farb-Margalit2012} that a non-zero element of $\mathrm{H}_1(S_g, \mathbb{Z})$ is primitive  if and only if  it is represented by an oriented simple closed curve. For the isotopy class $a$ of an oriented simple closed curve in $S_g$, we denote by $[a] \in \mathrm{H}_1(S_g, \mathbb{Z})$ its homology class. Since $\mathcal{D}_g^{ns}$ consists of isotopy classes of unoriented simple closed curves, there are two choices for the homology class $[a] \in \mathcal{P}_g'$  for each $a \in \mathcal{D}_g^{ns}$. We choose $[a]$ such that the entry in its first non-zero coordinate (from left) is positive. This gives a surjection   $\mathcal{D}_g^{ns} \to \mathcal{P}'_g$ given by $a \mapsto [a]$. Composing  the surjections $\mathcal{D}_g^{ns} \to \mathcal{P}'_g$, $\mathcal{P}'_g \to \mathcal{P}_g$ and $ \mathcal{P}_g \to \mathcal{P}_{g,n}$ gives a surjection $$\phi: \mathcal{D}_g^{ns} \to \mathcal{P}_{g,n},$$
which is a quandle homomorphism due to \cite[Theorem 7.1]{Dhanwani-Raundal-Singh-2021}. Let $(\dq_g^{ns})_n$ denote the $n$-quandle of the Dehn quandle $\dq_g^{ns}$ and $$\pi : \dq_g^{ns}\to (\dq_g^{ns})_n$$ the natural quotient map. 
 
 \begin{theorem}\label{dg-iso-pg2}
Let $\dq_g^{ns}$ be the Dehn quandle of the closed orientable surface $S_g$ of genus $g \ge 1$. Then the following hold:
\begin{enumerate}
\item $(\dq_g^{ns})_2 \cong \mathcal{P}_{g,2}$ for each $g \ge 1$.
\item $(\dq_2^{ns})_3 \cong \mathcal{P}_{2,3}$.
\end{enumerate}
\end{theorem}

\begin{proof}
Let $\phi: \dq_g^{ns} \to \mathcal{P}_{g,2}$ be the surjective quandle homomorphism as discussed above. Since $(\dq_g^{ns})_2$ is involutory, $\phi$ induces a surjective quandle homomorphism $$\overline{\phi}:(\dq_g^{ns})_2\to \mathcal{P}_{g,2}$$ 
such that $\overline{\phi} ~\pi= \phi$. We claim that $\overline{\phi}$  is an isomorphism. Suppose that $\phi(x)=\phi(y)$ for $x,y\in \dq_g^{ns}$. It suffices to prove that $\pi(x)=\pi(y)$. We identify elements of $\dq_g^{ns}$ with corresponding Dehn twists in $\mcg_g$ with quandle operation as conjugation. Since $\dq_g^{ns}$ is a connected quandle, we have $x=\tau_{a_1}*^{\epsilon_1}\tau_{z_1}*^{\epsilon_2}\tau_{z_2} * \cdots *^{\epsilon_k}\tau_{z_k}$ and $y=\tau_{a_1}*^{\delta_1}\tau_{w_1}*^{\delta_2}\tau_{w_2}* \cdots *^{\delta_r}\tau_{w_r}$ for some $z_i, w_j \in \dq_g^{ns}$ and 
$\epsilon_i, \delta_j \in \{ 1, -1 \}$, where $a_1$ denotes the usual curve (see Figure \ref{figure2}). Since $\phi(x)=\phi(y)$, we have 
$$\phi(\tau_{z_k}^{\epsilon_k}\cdots \tau_{z_2}^{\epsilon_2}\tau_{z_1}^{\epsilon_1}\tau_{a_1}\tau_{z_1}^{-\epsilon_1} \tau_{z_2}^{-\epsilon_2}\cdots \tau_{z_k}^{-\epsilon_k})= \phi(\tau_{w_r}^{\delta_r}\cdots \tau_{w_2}^{\delta_2}\tau_{w_1}^{\delta_1}\tau_{a_1}\tau_{w_1}^{-\delta_1} \tau_{w_2}^{-\delta_2}\cdots \tau_{w_r}^{-\delta_r}).$$
It follows from \cite[Proposition 7.6(ii)]{Dhanwani-Raundal-Singh-2021} that $\phi$ is simply the restriction of $\Psi_2:\mathcal{M}_g \to \text{Sp}(2g,\mathbb{Z}_2)$. Thus, the preceding equation can be written as 
$$\Psi_2(\tau_{z_k}^{\epsilon_k}\cdots \tau_{z_2}^{\epsilon_2}\tau_{z_1}^{\epsilon_1}\tau_{a_1}\tau_{z_1}^{-\epsilon_1} \tau_{z_2}^{-\epsilon_2}\cdots \tau_{z_k}^{-\epsilon_k})= \Psi_2(\tau_{w_r}^{\delta_r}\cdots \tau_{w_2}^{\delta_2}\tau_{w_1}^{\delta_1}\tau_{a_1}\tau_{w_1}^{-\delta_1} \tau_{w_2}^{-\delta_2}\cdots \tau_{w_r}^{-\delta_r}),$$
which implies that $\Psi_2(\tau_{w_1}^{-\delta_1}\tau_{w_2}^{-\delta_2} \cdots \tau_{w_r}^{-\delta_r} \tau_{z_k}^{\epsilon_k} \cdots \tau_{z_2}^{\epsilon_2} \tau_{z_1}^{\epsilon_1})$
commutes with $\Psi_2(\tau_{a_1})$. Using Corollary \ref{centraliser short exact mcg}, we can write $\tau_{w_1}^{-\delta_1}\tau_{w_2}^{-\delta_2} \cdots \tau_{w_r}^{-\delta_r} \tau_{z_k}^{\epsilon_k} \cdots \tau_{z_2}^{\epsilon_2} \tau_{z_1}^{\epsilon_1}=u v$ such that $\Psi_2(u)=I_{2g}$ and $v$ commutes with $\tau_{a_1}$ (by choosing an appropriate section to \eqref{congruence exact sequence}). By \cite[Proposition 2.1]{Humphries1992}, we know that $\ker(\Psi_2)$ is generated by squares of Dehn twists along non-separating curves. Thus, 
\begin{eqnarray*}
&& \pi(\tau_{a_1}*^{\epsilon_1}\tau_{z_1}*^{\epsilon_2}\tau_{z_2} * \cdots *^{\epsilon_k}\tau_{z_k}*^{-1}*^{-\delta_r} \tau_{w_r}* \cdots *^{-\delta_2} \tau_{w_2}*^{-\delta_1} \tau_{w_1})\\
&=&\pi(\tau_{a_1}*v*u)\\
&=&\pi(\tau_{a_1}*u), \quad \text{since $v$ commutes with $\tau_{a_1}$}\\
&=&\pi(\tau_{a_1})*\pi(u)\\
&=& \pi(\tau_{a_1}),
\end{eqnarray*}
which implies that $\pi(x)=\pi(y)$. This completes the proof of assertion (1).
\medskip

For assertion (2), we claim that kernel of $\Psi_3:\mathcal{M}_2 \to \text{Sp}(4,\mathbb{Z}_3)$ is generated by cubes of Dehn twists along non-separating curves. Let $N= \llangle \tau_{a_1}^3\rrangle$ be the normal closure in $\mcg_2$ of the set of cubes of Dehn twists along all non-separating simple closed curves on $S_2$. By \cite[Proposition 2.1]{Humphries1992}, $\ker(\Psi) \le N$. Let $M=\llangle I_4 + 3 E_{1,2}\rrangle$ be the normal closure of  $I_4 + 3 E_{1,2}$ in $\text{Sp}(4,\mathbb{Z})$ and $\psi_3:\text{Sp}(4,\mathbb{Z}) \to \text{Sp}(4,\mathbb{Z}_3)$ be the mod 3 reduction homomorphism. By \cite[Hilfssatz 9.2]{MR0181676}, $\ker(\psi_3)= M$. Hence, it follows that
$$N/\ker(\Psi) \cong \Psi(N)=M=\ker(\psi_3).$$
This gives
$$\mcg_2/N \cong \big(\mcg_2/\ker(\Psi)\big)/\big(N/ \ker(\Psi) \big) \cong \text{Sp}(4,\mathbb{Z})/ \ker(\psi_3) \cong \text{Sp}(4,\mathbb{Z}_3).$$
Since $\Psi_3= \psi_3 \Psi$, the preceding isomorphism is, in fact, induced by $\Psi_3$. Hence, it follows that $\ker(\Psi_3)=N= \llangle \tau_{a_1}^3\rrangle$. The proof of assertion (2) now follows analogously to that of assertion (1).
\end{proof}

\begin{remark}
Lemma \ref{gen-com-a1-in-modp} is inspired by \cite{MR4447971, DDPR}, wherein similar arguments have been used to understand the stabiliser of  a homological vector in the mapping class group $\mcg_g$, which also preserves the fiber over that homological vector under the homomorphism $\phi:  \mathcal{D}_g^{ns} \to \mathcal{P}_{g,n}$.
\end{remark}
\medskip

\subsection{$n$-quandles of Dehn quandles of surfaces}
We now present our main result, which can be thought of as the Dehn quandle analogue of a similar result \cite{MR3704243} of Hoste and Shanahan for link quandles. 

\begin{theorem}\label{finiteness n-quandle surface}
Let $\dq_g^{ns}$ be the Dehn quandle of the closed orientable surface of genus $g \ge 1$. Then the following hold:
\begin{enumerate}
\item $(\dq_1^{ns})_n$ is finite if and only if $n=2,3,4,5$.
\item $(\dq_2^{ns})_n$ is finite if and only if $n=2,3$.
\item If $g\geq 3$, then $(\dq_g^{ns})_n$ is finite for $n=2$ and infinite for $n \neq 2, 3, 4, 6, 8, 12$.
\end{enumerate}
\end{theorem}

\begin{proof}
It follows from \cite{NiebrzydowskiPrzytycki2009} that $\dq_1^{ns} \cong Q(3_1)$, the knot quandle of the trefoil knot $3_1$. Further, it follows from \cite[Section 5]{MR3704243} that $(Q(3_1))_n$ is finite if and only if $n=2,3,4,5$, which proves assertion (1).
\par

By Theorem \ref{dg-iso-pg2}, we obtain that $(\dq_2^{ns})_n$ is finite for $n=2, 3$. Let $\mathcal{M}_g^b$ be the mapping class group of the surface $S_g^b$ of genus $g$ with $b \ge 0$ boundary components and let $(\mathcal{M}_g^b)_n$ be the quotient of $\mathcal{M}_g$ by the normal closure of $n$-th powers of Dehn twists along non-separating simple closed curves.  We avoid writing $b$ in the notation whenever $b=0$.  It is proved in \cite[Theorem 4]{Humphries1992} that $(\mathcal{M}_2^b)_n$ is infinite for $b \ge0$ and $n>3$. Taking $b=0$, it follows from Theorem  \ref{finiteness of g_n} that  $(\dq_2^{ns})_n$ is infinite for $n>3$, which proves assertion (2).
\par

Assume that $g \ge 3$. By Theorem \ref{dg-iso-pg2},  $(\dq_g^{ns})_2$ is finite. On the other hand, \cite[Corollary 1.2]{Funar1999} gives that $(\mathcal{M}_g)_n$ is infinite for $n \neq 2, 3, 4, 6, 8, 12$. Thus,  Theorem \ref{finiteness of g_n} implies that $(\dq_g^{ns})_n$ is infinite for $n \neq 2, 3, 4, 6, 8, 12$. 
\end{proof}

\begin{remark}
Theorem \ref{finiteness n-quandle surface} strengthens and generalises \cite[Theorem 8.4]{Dhanwani-Raundal-Singh-2021} wherein only the finiteness of $(\dq_g^{ns})_2$ is proved for $g = 1, 2, 3$. We do not know whether $(\dq_g^{ns})_n$ is finite for $g\geq 3$ and $n=3, 4, 6, 8, 12$.
\end{remark}

We conclude with a bound on the size of the smallest non-trivial quandle quotient of the Dehn quandle of a surface. The proof is based on \cite[Lemma 11]{kolay21}.

\begin{proposition}\label{smallest quotient of dehn surface}
Let $g\geq 1$ and $f:\dq_g^{ns} \to Q$ be a surjective quandle homomorphism onto a quandle with at least two elements. Then $|Q|\geq 2^{2g}-1$, and the bound is sharp.
\end{proposition}

\begin{proof}
By \cite[Theorem 7.1]{Dhanwani-Raundal-Singh-2021}, there is a surjective quandle homomorphism from $\mathcal{D}_g^{ns}$ onto the projective primitive homological quandle $\mathcal{P}_{g,2}$ of order $2^{2g}-1$. It follows from \cite[Proposition 6.2]{Farb-Margalit2012} that a non-zero element of $\mathrm{H}_1(S_g,\mathbb{Z}_2)$ is primitive if and only if it is represented by an oriented simple closed curve on $S_g$. Following \cite[Section 5]{kolay21}, we suitably choose simple closed curves on $S_g$ such that they represent all primitive vectors in $\mathrm{H}_1(S_g,\mathbb{Z}_2)$. Let $X$ be the set of isotopy classes of (right) Dehn twists about these suitably chosen curves.  Viewing $\dq_g^{ns}$ as a subquandle of the conjugation quandle of $\mcg_g$, let $f:\dq_g^{ns} \to Q$ be a surjective quandle homomorphism. Suppose that $f(\tau_\alpha) =f(\tau_\beta)$ for some $\tau_\alpha, \tau_\beta \in X$ with $\tau_\alpha \neq \tau_\beta$. Then, by the construction of $X$,  can find a simple closed curve $\gamma$ such that $\tau_\gamma$ commutes with precisely one of $\tau_\alpha$ or $\tau_\beta$ and satisfies the braid relation with the other. Without loss of generality, suppose that 
$$\tau_\gamma \tau_\alpha \tau_\gamma= \tau_\alpha \tau_\gamma \tau_\alpha \quad \textrm{and} \quad \tau_\gamma \tau_\beta= \tau_\beta\tau_\gamma.$$
Applying $f$ to the preceding relations and using the equality $f(\tau_\alpha) =f(\tau_\beta)$, we obtain $f(\tau_\alpha) =f( \tau_\gamma)$. Proceeding in this manner, we show that $f$ is constant on $X$. Since $X$ also generates $\dq_g^{ns}$ as a quandle \cite[Proposition 3.2]{Dhanwani-Raundal-Singh-2021}, it follows that $f$ is constant on $\dq_g^{ns}$, which is a contradiction. Hence, $f$ must be injective on $X$, and therefore $|Q|\geq 2^{2g}-1$.
\end{proof}
\medskip

\noindent\textbf{Acknowledgments.}
The authors thank Prof. Louis Funar for pointing out an oversight in their earlier proof of assertion (3) of Theorem \ref{finiteness n-quandle surface}. Some of the ideas in this paper were developed by the authors during the preparation of \cite{Dhanwani-Raundal-Singh-2021, Dhanwani-Raundal-Singh-2022}. They thank Dr. Hitesh Raundal for his interest in the ideas. Dhanwani thanks IISER Mohali for the institute postdoctoral fellowship. Singh is supported by the Swarna Jayanti Fellowship grants DST/SJF/MSA-02/2018-19 and SB/SJF/2019-20/04.

\end{document}